\documentclass[12pt]{amsart}

\usepackage{amsfonts, amssymb}
\usepackage{verbatim}       
\usepackage{epsfig}
\usepackage{graphicx}

\usepackage{galois}
\usepackage{amsmath} \allowdisplaybreaks[4]

\newtheorem{theorem}{Theorem}

\newtheorem{proposition}{Proposition}

\newtheorem{remark}{Remark}

\def\eps{\varepsilon}

\graphicspath{{converted_graphics/}}
\begin{document}

\title[Instability of Isolated Spectrum for W-shaped Maps]{ Instability of Isolated Spectrum for W-shaped Maps}
\thanks{The research of the authors was supported by NSERC grants. }
\author{Zhenyang Li}
\address{Department of Mathematics and Statistics, Concordia University,
1455 de Maisonneuve Blvd. West, Montreal, Quebec H3G 1M8, Canada}
\email{zhenyangemail@gmail.com}
\email{pgora@mathstat.concordia.ca}
\author{Pawe\l\ G\'{o}ra}

\date{\today }

\begin{abstract}
In this note we consider  $W$-shaped map  $W_0=W_{s_1,s_2}$ with $\frac 1{s_1}+\frac 1{s_2}=1$ and show that  eigenvalue $1$ is not stable.
We do this in a constructive way. For each perturbing map $W_a$ we show the existence of the ``second" eigenvalue $\lambda_a$,
such that   $\lambda_a\to 1$, as $a\to 0$, which proves instability of isolated spectrum of $W_0$.
At the same time, the existence of  second eigenvalues close to 1 causes the maps $W_a$ behave in a metastable way. They have two almost invariant sets
and the system spends long periods of consecutive iterations in each of them with infrequent jumps from one to the other.
\end{abstract}
\maketitle

\section{Introduction}
One of the most important problems in the theory of dynamical systems is their stability and possible instability.
In particular, in the theory of piecewise expanding maps of interval, it is interesting whether the given system
has a stable absolutely continuous invariant measure (acim), and more generally, if the isolated spectrum of
Perron-Frobenius operator is stable under small perturbations of the map. For a general introduction to the theory of
 piecewise expanding one-dimensional maps we refer the reader to
\cite{BG}. Most relevant to the   stability problems are papers \cite{K} and \cite{KL99}.

In general, the setting of the stability problem we are interested in is as follows: Let $\tau_0$ be a piecewise expanding map of an interval
with unique acim $\mu_0$ and $\{\tau_a\}_{a>0}$ a family of its perturbations with acims $\mu_a$, correspondingly. If maps $\tau_a$ converge
to $\tau_0$ (say, in Skorokhod metric), do their acims converge (say, in $*$-weak topology) to $\mu_0$? Or more generally, do the isolated spectra
of $P_{\tau_a}$ converge to the isolated spectrum of $P_{\tau_0}$, including multiplicities and eigenfunctions? $P_\tau$ is the Perron-Frobenius operator induced
by $\tau$ on the space of functions of bounded variation and by isolated spectrum we mean the part of the spectrum which lies outside the essential spectral radius.
Papers  \cite{K} and \cite{KL99} show that such stability takes place  if the family $\{\tau_a\}_{a\ge 0}$ satisfies Lasota-Yorke inequality (\cite{LaY} or \cite{EG} for
strengthened form) with uniform constants. Usual conditions ensuring this are $|\tau_a'|> 2+\eps$ plus the minimal length of subintervals of defining partitions
uniformly separated from $0$.

One of the known sources of instability  is the presence of turning fixed or periodic  point touching  a map branch  with slope 2 or smaller.
The famous example is the $W$-shaped map introduced in \cite{K}. Because of the turning fixed point we cannot use an iterate of the map to increase
the minimal slope. It causes appearance of arbitrary short partition intervals in perturbed maps.

Recently the interest in $W$-shaped maps increased. We will introduce them in more detail. The $W_{s_1,s_2}$ map is a piecewise linear map of the interval $[0,1]$ onto itself with a graph in the shape of letter W. The first and the third branch are decreasing, the second and the last increasing.
The first and the last branches are onto and with relatively large slopes (usually around  -4 and 4). The second branch has slope $s_1$ and the third
has slope $-s_2$ and $1/2$ is the point where they meet. $W_{s_1,s_2}(1/2)=1/2$ so $1/2$ is  the turning fixed point. The original $W$-map of
\cite{K} is of $W_{2,2}$ type and it was proved there that its acim is unstable under some family of localized perturbations.
In \cite{EM} a family of non-local Markov perturbations was constructed which also caused  the instability of acim of $W_{2,2}$. By non-local we mean that each perturbed map is exact on the whole $[0,1]$. This result has been generalized in \cite{LGBPE} where a whole continuous family of perturbations
was constructed with the same effect. More general situation was considered in \cite{ZhLi}. The perturbations similar to that of
\cite{LGBPE} were considered.  It was shown there that depending on whether $\frac 1{s_1}+\frac 1{s_2}$ is larger, equal or smaller than  $1$ the limit of $\mu_a$'s is Dirac measure $\delta_{1/2}$, or a combination of $\delta_{1/2}$ and $\mu_0$, or $\mu_0$, correspondingly. This result suggested that
condition $\frac 1{s_1}+\frac 1{s_2}<1$ may actually imply stability which was later proved for a quite general setting in \cite{EG}.

In this note we consider map $W_0$ of type  $W_{s_1,s_2}$ with $\frac 1{s_1}+\frac 1{s_2}=1$ and show that that eigenvalue $1$ is not stable.
We do this in a constructive way. For each perturbed map $W_a$ we show the existence of the ``second" eigenvalue $\lambda_a$,
$$\frac{1-2ra}{1+2ra}<\lambda_a<\frac 1{1+2ra}\ ,$$
where $r$ is a constant independent of $a$. Thus, as $a\to 0$ the eigenvalues $\lambda_a\to 1$ which shows instability of isolated spectrum of $W_0$.
At the same time, the existence of  second eigenvalues close to 1 causes the maps $W_a$ behave in a metastable way. They have two almost invariant sets
and the system spends long periods of consecutive iterations in each of them with infrequent jumps from one to the other.

\section{Markov $W_a$ maps and their invariant densities}

Let $s_1$, $s_2>1$ satisfy $\frac{1}{s_1}+\frac{1}{s_2}=1$, and $r>0$. Let us consider  $W$-shaped map:
\begin{eqnarray*}
W_0(x) = \begin{cases}
W_{0,1}(x):=1-2s_2x, & 0\leq x < \frac{1}{2}-\frac{1}{2s_1}, \\
W_{0,2}(x):=s_1(x-\frac{1}{2}+\frac{1}{2s_1}), & \frac{1}{2}-\frac{1}{2s_1}\leq x < \frac{1}{2},\\
W_{0,3}(x):=s_2(\frac{1}{2}+\frac{1}{2s_2}-x), & \frac{1}{2}\leq x < \frac{1}{2}+\frac{1}{2s_2},\\
W_{0,4}(x):=2s_1(x-1)+1,& \frac{1}{2}+\frac{1}{2s_2}\leq x <1,
\end{cases}
\end{eqnarray*}
and its perturbations, $W_a$ maps with parameter $a>0$:
\begin{eqnarray*}
W_a(x) = \begin{cases}
W_{a,1}(x):=1-2s_2x, & 0\leq x < \frac{1}{2}-\frac{1}{2s_1}, \\
W_{a,2}(x):=(s_1+2rs_1a)(x-\frac{1}{2}+\frac{1}{2s_1}), & \frac{1}{2}-\frac{1}{2s_1}\leq x < \frac{1}{2},\\
W_{a,3}(x):=(s_2+2rs_2a)(\frac{1}{2}+\frac{1}{2s_2}-x), & \frac{1}{2}\leq x < \frac{1}{2}+\frac{1}{2s_2},\\
W_{a,4}(x):=2s_1(x-1)+1,& \frac{1}{2}+\frac{1}{2s_2}\leq x <1.
\end{cases}
\end{eqnarray*}

Let $\tau_i=W_{a,i}^{-1}$, $i=1, 2, 3, 4$; $I_0=[0,\frac{1}{2}+ra]$. The Frobenius-Perron operator (see \cite{BG}) associated with $W_a$ is
\begin{eqnarray*}
P_af =\frac{1}{2s_2}f\comp \tau_1+\frac{1}{s_1+2rs_1a}(f\comp \tau_2)\chi_{I_0} &+&\frac{1}{s_2+2rs_2a}(f\comp \tau_3)\chi_{I_0}\\&+&\frac{1}{2s_1}f\comp \tau_4
\end{eqnarray*}
Note that
\begin{equation}\label{Eq:iterate}\begin{split}
&\chi_{I_0}\comp\tau_1=1\ \ , \ \ \chi_{I_0}\comp \tau_2=\chi_{I_0}\ , \\
& \chi_{I_0}\comp \tau_3=\chi_{[W_a^2(1/2),\frac{1}{2}+ra]}\  \ , \ \ \chi_{I_0}\comp \tau_4=0\ .
\end{split}
\end{equation}
Let $I_1=[W_a^2(1/2),\frac{1}{2}+ra]$ whose left end point is $W^2_a(\frac{1}{2})=W_a(\frac{1}{2}+ra)$.

We will consider only parameters  $a$ such that $W_a$ is a Markov map, i.e., some iterate of $1/2$ falls into an endpoint  of the defining partition.
Let $a$ satisfy:
\begin{equation}\label{Eq:def m1}
W_a^{m}(\frac{1}{2}+ra)=\frac{1}{2}-\frac{1}{2s_1},
\end{equation}
where $m \geq 1$ is the first time when the trajectory of  $W_a(\frac{1}{2})=\frac{1}{2}+ra$ reaches the partition point $\frac{1}{2}-\frac{1}{2s_1}$. Note that $\frac{1}{2}-\frac{1}{2s_1}=\frac{1}{2s_2}$. The point $W_a(\frac{1}{2}+ra)$ is just below the fixed point on the second branch of $W_a$ and the consecutive images
$W_a^{i}(\frac{1}{2}+ra)$ decrease until for $i=m$ the equality (\ref{Eq:def m1}) is satisfied.

Let us take $1$ as the initial function  and iterate it using operator $P_a$. Let $P_a^n 1$  be denoted by $f_{n,m}$. Let \[I_i=[W_a^i(\frac{1}{2}+ra),\frac{1}{2}+ra], i=1,2,\cdots, m.\]

 Because of (\ref{Eq:def m1})and (\ref{Eq:iterate}), after some number of  iterations ($n\ge m+1$), we have:
\[f_{n,m}=c_{n,0}+\alpha_{n,0}\chi_{I_0}+\alpha_{n,1}\chi_{I_1}+\alpha_{n,2}\chi_{I_2}+\cdots+\alpha_{n,m-1}\chi_{I_{m-1}}+\alpha_{n,m}\chi_{I_{m}},\]
where $c_{n,0}$ and $\alpha_{n,i}$ ($i=0,1,\cdots,m$) are constants. Now, let us look at the $f_{n+1,m}$. We have the following proposition.

\begin{proposition}\label{Prop4}
$(I)$ $c_{n,0}\comp \tau_1$ and $c_{n,0}\comp \tau_4$ are again constant functions, $c_{n,0}\comp \tau_2\chi_{I_0}$ and $c_{n,0}\comp \tau_3\chi_{I_0}$ are the characteristic function $\chi_{I_0}$;

$(II)$ $\chi_{I_0}\comp \tau_1$ is constant function, $\chi_{I_0}\comp \tau_2\chi_{I_0}=\chi_{I_0}$, $\chi_{I_0}\comp \tau_3\chi_{I_0}=\chi_{I_1}$, $\chi_{I_0}\comp \tau_4$ is 0;

$(III)$ For $i=1, 2, \cdots, m-1$, $\chi_{I_i}\comp \tau_1$ and $\chi_{I_i}\comp \tau_4$ are 0, $\chi_{I_i}\comp \tau_2\chi_{I_0}=\chi_{I_{i+1}}$, $\chi_{I_i}\comp \tau_3\chi_{I_0}=\chi_{I_1}$;

$(IV)$ $\chi_{I_m}\comp \tau_1$ and $\chi_{I_m}\comp \tau_4$ are 0, $\chi_{I_m}\comp \tau_2\chi_{I_0}=\chi_{I_0}$, $\chi_{I_m}\comp \tau_3\chi_{I_0}=\chi_{I_1}$.
\end{proposition}

Thus, we have the following proposition.
\begin{proposition}\label{Prop5}
 for $n$ big enough, $f_{n,m}$ always has the form:
$$f_{n,m}=c_{n,0}+\alpha_{n,0}\chi_{I_0}+\alpha_{n,1}\chi_{I_1}+\alpha_{n,2}\chi_{I_2}+\cdots+\alpha_{n,m-1}\chi_{I_{m-1}}+\alpha_{n,m}\chi_{I_{m}},$$
and
\begin{eqnarray*}
\left[ \begin {array}{c} c_{n+1,0}\\
\alpha_{n+1,0}\\
\alpha_{n+1,1}\\
\vdots\\
\alpha_{n+1,m}\end {array} \right]=A_m\left[ \begin {array}{c} c_{n,0}\\
\alpha_{n,0}\\
\alpha_{n,1}\\
\vdots\\
\alpha_{n,m}\end {array} \right]\ .
\end{eqnarray*}
where $(m+2)\times(m+2)$ matrix $A_m$ is given by
{\tiny
\begin{eqnarray*}
A_m=\left[ \begin {array}{cccccccc} \frac{1}{2s_1}+\frac{1}{2s_1}&\frac{1}{2s_2}&0&0&0&\cdots&0&0\\ \noalign{\medskip} \frac{1}{s_1+2rs_1a}+\frac{1}{s_2+2rs_2a}& \frac{1}{s_1+2rs_1a}&0&0&0&\cdots&0& \frac{1}{s_1+2rs_1a}\\ \noalign{\medskip}0&
 \frac{1}{s_2+2rs_2a}& \frac{1}{s_2+2rs_2a}& \frac{1}{s_2+2rs_2a}& \frac{1}{s_2+2rs_2a}& \cdots&
 \frac{1}{s_2+2rs_2a}& \frac{1}{s_2+2rs_2a}
\\ \noalign{\medskip}0&0& \frac{1}{s_1+2rs_1a}&0&0&\cdots&0&0
\\ \noalign{\medskip}0&0&0& \frac{1}{s_1+2rs_1a}&0&\cdots&0&0
\\ \noalign{\medskip}\vdots&\vdots&\vdots&\vdots&\vdots&\ddots&\vdots&\vdots
\\ \noalign{\medskip}0&0&0&0&0& \cdots&0&0
\\ \noalign{\medskip}0&0&0&0&0&\cdots& \frac{1}{s_1+2rs_1a}&0
\end {array} \right]\ .
\end{eqnarray*}
}
\end{proposition}

Since $\frac{1}{s_1}+\frac{1}{s_2}=1$, we can simplify the $A_m$ to the form
\begin{eqnarray*}
A_m=\left[ \begin {array}{cccccccc} \frac{1}{2}&\frac{1}{2s_2}&0&0&0&\cdots&0&0\\ \noalign{\medskip} \frac{1}{1+2ra}& \frac{1}{s_1+2rs_1a}&0&0&0&\cdots&0& \frac{1}{s_1+2rs_1a}\\ \noalign{\medskip}0&
 \frac{1}{s_2+2rs_2a}& \frac{1}{s_2+2rs_2a}& \frac{1}{s_2+2rs_2a}& \frac{1}{s_2+2rs_2a}& \cdots&
 \frac{1}{s_2+2rs_2a}& \frac{1}{s_2+2rs_2a}
\\ \noalign{\medskip}0&0& \frac{1}{s_1+2rs_1a}&0&0&\cdots&0&0
\\ \noalign{\medskip}0&0&0& \frac{1}{s_1+2rs_1a}&0&\cdots&0&0
\\ \noalign{\medskip}\vdots&\vdots&\vdots&\vdots&\vdots&\ddots&\vdots&\vdots
\\ \noalign{\medskip}0&0&0&0&0& \cdots&0&0
\\ \noalign{\medskip}0&0&0&0&0&\cdots& \frac{1}{s_1+2rs_1a}&0
\end {array} \right].
\end{eqnarray*}
We also need the following proposition.
\begin{proposition}
Equation (\ref{Eq:def m1}) is equivalent to:
\[(s_2+2rs_2a)(s_1+2rs_1a)^{m-1}-\sum\limits_{i=0}^{m-1}(s_1+2rs_1a)^i=\frac{1}{2rs_1a}\ ,\]
or
\begin{equation}\label{good m1}
(s_1+2rs_1a)^m=\frac{1}{4r^2s_2^2a^2}.
\end{equation}
\end{proposition}
\begin{proof}
If \[(s_2+2rs_2a)(s_1+2rs_1a)^{m-1}-\sum\limits_{i=0}^{m-1}(s_1+2rs_1a)^i=\frac{1}{2rs_1a}\ ,\]
then
\[(s_1+2rs_1a)^{m-1}\left[(s_2+2rs_2a)(s_1+2rs_1a-1)-(s_1+2rs_1a)\right]=\frac{s_1-1}{2rs_1a}=\frac{1}{2rs_2a},\]
Thus, we obtain
\[(s_1+2rs_1a)^m=\frac{1}{4r^2s_2^2a^2}.\]

On the other hand, it is proven in \cite{ZhLi} that
\begin{eqnarray*}
&&W_a^{m+1}(1/2)\\
&&=-a^2(s_1+2rs_1a)^{m-1}\frac{r(2rs_1s_2+2rs_1s_2-2rs_1-2rs_2)+4r^3s_1s_2a}{s_1+2rs_1a-1}\\
&&\hskip 6 cm +\frac{s_1-1+2rs_1a-2ra}{2(s_1+2rs_1a-1)}.
\end{eqnarray*}
If equation (\ref{Eq:def m1}) holds, then
\[a^2(s_1+2rs_1a)^{m-1}\frac{2r^2s_1s_2+4r^3s_1s_2a}{s_1+2rs_1a-1}=\frac{s_1-1}{2s_1(s_1+2rs_1a-1)},\]
hence,
\[(s_1+2rs_1a)^{m-1}=\frac{1}{a^24r^2s_1s_2^2(1+2ra)}\]
which is equivalent to equation (\ref{good m1}).
\end{proof}
Using Proposition \ref{Prop5}, we can find the fixed vector of $A_m$. Let us denote it by $(c, \alpha_0, \alpha_1, \cdots, \alpha_{m})^T$. Then, the fixed function
(not necessarily normalized) of $P_a$ is:
$$g_m^*=c+\alpha_0\chi_{I_0}+\alpha_1\chi_{I_1}+\alpha_2\chi_{I_2}+\cdots+\alpha_{m-1}\chi_{I_{m-1}}+\alpha_{m}\chi_{I_{m}},$$
where
\begin{eqnarray*}
c&=&\frac{1}{2rs_1s_2a}\\
\alpha_0&=&\frac{1}{2rs_1a}\\
\alpha_1&=&(s_1+2rs_1a)^{m-1}\\
\alpha_2&=&(s_1+2rs_1a)^{m-2}\\
\cdots&&\\
\alpha_{m-2}&=&(s_1+2rs_1a)^{2}\\
\alpha_{m-1}&=&s_1+2rs_1a\\
\alpha_{m}&=&1.
\end{eqnarray*}

It was proven in \cite{ZhLi} that after normalization the  measures $g_m^*\cdot L$ converge to the measure
$$ \frac1{2r(s_1+s_2)(s_2+2)+2rs_1s_2^2}\left({2r(s_1+s_2)(s_2+2)}\mu_0+{2rs_1s_2^2}\delta_{1/2}\right) \ ,$$
where $L$ is Lebesgue measure, $\mu_0$ is the absolutely continuous invariant measure of $W_0$ and $\delta_{1/2}$ is Dirac measure at $1/2$.

\section{Second eigenvalues for Markov $W_a$ maps}
Now, instead for a fixed vector, we will look for an eigenvector corresponding to an eigenvalue $\lambda<1$.
Denote the eigenvector of $A_m$ associated with $\lambda$  by $(c, \alpha_0, \alpha_1, \cdots, \alpha_{m})^T$. Then, the corresponding eigenfunction  of $P_a$ associated with $\lambda$ is:
\begin{equation}\label{def g}
h_m=c+\alpha_0\chi_{I_0}+\alpha_1\chi_{I_1}+\alpha_2\chi_{I_2}+\cdots+\alpha_{m-1}\chi_{I_{m-1}}+\alpha_{m}\chi_{I_{m}}\ .
\end{equation}
The equation $$A_m h_m=\lambda h_m\ ,$$ is equivalent to the system
\begin{eqnarray*}
\lambda c&=&\frac{1}{2} c+\frac {1}{2s_2} \alpha_0\\
\lambda \alpha_0&=&\frac{1}{1+2ra}(c+\frac{1}{s_1}\alpha_0+\frac{1}{s_1}\alpha_m)\\
\lambda \alpha_1&=&\frac{1}{s_2(1+2ra)}(\alpha_0+\alpha_1+\dots+\alpha_m)\\
\lambda \alpha_2&=&\frac{1}{s_1(1+2ra)}\alpha_1\\
\lambda \alpha_3&=&\frac{1}{s_1(1+2ra)}\alpha_2\\
\cdots&&\\
\lambda \alpha_{m-2}&=&\frac{1}{s_1(1+2ra)}\alpha_{m-3}\\
\lambda \alpha_{m-1}&=&\frac{1}{s_1(1+2ra)}\alpha_{m-2}\\
\lambda \alpha_{m}&=&\frac{1}{s_1(1+2ra)}\alpha_{m-1}.
\end{eqnarray*}
We can solve this system starting from the last equation. Let $\alpha_m=1$. Then,
\begin{equation}\label{eigenvector1}\begin{split}
\alpha_{m}&=1 \\
\alpha_{m-1}&=\lambda s_1(1+2ra)\\
\alpha_{m-2}&=\lambda^2s_1^2(1+2ra)^{2}\\
\cdots&\\
\alpha_2&=\lambda^{m-2}s_1^{m-2}(1+2ra)^{m-2}\\
\alpha_1&=\lambda^{m-1}s_1^{m-1}(1+2ra)^{m-1}\\
\alpha_0&= \lambda s_2(1+2ra)\alpha_1-(\alpha_1+\alpha_2+\dots+\alpha_m)\\
&=\lambda^{m}s_1^{m-1}s_2(1+2ra)^{m}-\frac{\lambda^{m}s_1^{m}(1+2ra)^{m}-1}{\lambda s_1(1+2ra)-1}\\
&=\frac{\lambda^{m}s_1^{m-1}(1+2ra)^{m}(\lambda s_1s_2(1+2ra)-s_1s_2)+1}{\lambda s_1(1+2ra)-1}\\
c&=\lambda(1+2ra)\alpha_0-\frac{\alpha_0}{s_1}-\frac{1}{s_1}\\
c&=\frac {\alpha_0}{s_2(2\lambda-1)}
\end{split}
\end{equation}
We have two expressions  for $c$. The system can be solved only if they are equal. Thus, we obtain equation
\begin{equation*}\begin{split}
&\lambda^{m}s_1^{m-2}(1+2ra)^{m}(\lambda s_1s_2(1+2ra)-s_1s_2)\\
&\hskip 2cm = \frac{\lambda^{m}s_1^{m-1}(1+2ra)^{m}(\lambda s_1s_2(1+2ra)-s_1s_2)+1}{s_2(2\lambda-1)(\lambda s_1(1+2ra)-1)}\ ,
\end{split}
\end{equation*}
or
$$\lambda^{m}s_1^{m-1}s_2(1+2ra)^{m}(\lambda (1+2ra)-1)\left[s_2(2\lambda-1)(\lambda s_1(1+2ra)-1)-s_1\right]=1\ .$$
We are going to prove that for small $a$ this equation has a solution  $\frac {1-2ra}{1+2ra}<\lambda<\frac 1{1+2ra}$.
Let us introduce an auxiliary function
 $$\phi(\lambda)=\lambda^{m}s_1^{m-1}s_2(1+2ra)^{m}(\lambda (1+2ra)-1)\left[s_2(2\lambda-1)(\lambda s_1(1+2ra)-1)-s_1\right]\ .$$
Obviously $\phi(\frac 1{1+2ra})=0$. We will show that $\phi(\frac {1-2ra}{1+2ra})>1$ if $a$ is small enough.
We have
\begin{eqnarray*}
\phi\left(\frac {1-2ra}{1+2ra}\right)&=&\left(\frac {1-2ra}{1+2ra}\right)^{m}s_1^{m-1}s_2\left(1+2ra\right)^{m}\left(\left(\frac {1-2ra}{1+2ra}\right) \left(1+2ra\right)-1\right)\cdot\\
&&\quad \cdot\left[s_2\left(2\left(\frac {1-2ra}{1+2ra}\right)-1\right)\left(\left(\frac {1-2ra}{1+2ra}\right) s_1\left(1+2ra\right)-1\right)-s_1\right]\\
&=& (1-2ra)^ms_1^{m-1}s_2(-2ra)\frac{-2ra(s_2+5s_1-6s_1s_2ra)}{1+2ra}\\
&=&(1-2ra)^ms_1^{m-1}s_24r^2a^2\frac{s_2+5s_1-6s_1s_2ra}{1+2ra}.
\end{eqnarray*}
Using (\ref{good m1}) we obtain
\begin{eqnarray*}
\phi\left(\frac {1-2ra}{1+2ra}\right)&=&\left(\frac{1-2ra}{1+2ra}\right)^m\frac{s_2+5s_1-6s_1s_2ra}{s_1s_2(1+2ra)}\\
&=& \left(\frac{1-2ra}{1+2ra}\right)^m\frac{1+\frac{4}{s_2}-6ra}{1+2ra}.
\end{eqnarray*}
Note that if $a<\frac{1}{2rs_2}$, then $\frac{1+\frac{4}{s_2}-6ra}{1+2ra}>1$. Furthermore, $\lim\limits_{a\rightarrow 0}\frac{1+\frac{4}{s_2}-6ra}{1+2ra}=1+\frac{4}{s_2}>1$.

Using (\ref{good m1}) again, we can represent $m$ as
\begin{equation}\label{m=}
m=\frac{-2\ln(2s_2ra)}{\ln(s_1+2s_1ra)}\ ,
\end{equation}
which gives
\begin{eqnarray*}
\phi\left(\frac{1-2ra}{1+2ra}\right)&=&\frac{1+\frac{4}{s_2}-6ra}{1+2ra}\left(\frac{1-2ra}{1+2ra}\right)^{\frac{-2\ln(2s_2ra)}{\ln(s_1+2s_1ra)}}\\
&=&\frac{1+\frac{4}{s_2}-6ra}{1+2ra}\exp\left(-2\ln\left(\frac{1-2ra}{1+2ra}\right)\frac{\ln(2s_2ra)}{\ln(s_1+2s_1ra)}\right)\ .
\end{eqnarray*}
Since $\lim\limits_{a\rightarrow 0}\ln\left(\frac{1-2ra}{1+2ra}\right)\ln(a)=0$,
the argument of $\exp$ converges to $0$ as $a\to 0$.  Thus
$$\lim_{a\to 0}\phi\left(\frac{1-2ra}{1+2ra}\right)=1+\frac{4}{s_2}\ .$$
This proves our claim for $a$ small enough.
We proved the following
\begin{theorem}\label{lambda}
Assume that  $a$ satisfies (\ref{good m1}), for  some integer $m$, i.e., $W_a$ map is Markov with $W_a^{m+1}(1/2)=\frac12-\frac{1}{2s_1}$.
For $a$ small enough, Perron-Frobenius operator $P_a$ has an eigenvalue $\lambda_a$ satisfying
\begin{equation}\label{lambda_est}
\frac{1-2ra}{1+2ra}<\lambda_a<\frac 1{1+2ra}\ .
\end{equation}
The corresponding eigenfunction is given by equations (\ref{def g}) and (\ref{eigenvector1}), up to a multiplicative constant.
\end{theorem}
\begin{remark}
Using tedious calculations we were able to show that $\phi''$ is positive in a neighbourhood of 1. Since
$\phi\left(\frac{1-2ra}{1+2ra}\right)>1$, $\phi\left(\frac{1}{1+2ra}\right)=0 $ and $\phi(1)=1$, for small $a$ there is only one
eigenvalue  in the interval $\left(\frac{1-2ra}{1+2ra},1\right)$, i.e., $\lambda_a$ we found in Theorem \ref{lambda} is really the ``second" eigenvalue.
\end{remark}

\section{Eigenfunction for $\lambda_a<1$}
In this section we take a closer look at the eigenfunction corresponding to the second  eigenvalue $\lambda_a$ found in Theorem \ref{lambda}. We omit  the subscript
``a" to simplify the notation.
Let $(c, \alpha_0, \alpha_1, \cdots, \alpha_{m})$ be the $\lambda$-eigenvector given by  (\ref{eigenvector1}).
We have $$\alpha_j=\lambda^{m-j}s_1^{m-j}(1+2ra)^{m-j}>0\ \ ,\ \ j=1,\dots, m\ .$$
Next, $$\alpha_0=\frac{\lambda^{m}s_1^{m-1}(1+2ra)^{m}(\lambda s_1s_2(1+2ra)-s_1s_2)+1}{\lambda s_1(1+2ra)-1}<0\ ,$$
since $\lambda(1+2ra)<1$ but very close to $1$ and using formula (\ref{m=}) we can show that $\lambda^{m}(1+2ra)^{m}$ approaches $1$ as $m\to\infty$ .
As $\alpha_0<0$, we also have
$$c=\frac {\alpha_0}{s_2(2\lambda-1)}<0 \ .$$
The $P_a$ eigenfunction $h_m$, defined in (\ref{def g}), is positive on some interval $G_m=[W_a^{m_1}(1/2),1/2+a/4]$ and negative
outside this interval. Since, as $a$ decreases, more and more of numbers $\alpha_m, \alpha_{m-1},\alpha_{m-2}\dots$
are necessary to balance $\alpha_0+c$, we have $\lim_{a\to 0}(m-m_1)=+\infty$. This implies, that intervals $G_m$ shrink to
the point $1/2$ as $a\to 0$.

Since $0<\lambda<1$  we  have $\int_0^1 h_m d{\rm L}=0$. Let $K_m=\int_0^1 |h_m| d{\rm L}$.
The normalized signed measures $\frac 1{K_m} h_m \cdot {\rm L}$ converge $*$-weakly to
the measure
$$-\frac 12 \mu_0+ \frac 12 \delta_{(1/2)}\ ,$$
where $\mu_0$ is $W_0$-invariant absolutely continuous measure and $\delta_{(1/2)}$ is  Dirac measure at point $1/2$.

As it is described in \cite{FrSt} the presence of the eigenvalue $\lambda$ close to $1$ makes the system behave in a metastable
way. The sets $A^+=\{t: h_m(t)\ge 0\}$ and $A^-=\{t: h_m(t)<0\}$ are almost invariant with the escape rates bounded by $-\ln \lambda$
which is close to $0$. This means that a typical trajectory stays for a long time in $A^+$, then jumps to $A^-$,
 stays there for a long time, then jumps to $A^+$, spend there
long time, etc. Despite the small essential spectral radius (equal to $\max\{1/s_1, 1/s_2\}$), the system converges to equilibrium slowly
at the rate given by $C\lambda^n$, for some constant $C$.

Fig. \ref{fig:eigenfunctions} shows graphs of normalized functions $h_m$ produced using Maple 13. We used  $s_1=s_2=2$ and $r=1/4$.

a) $m=5$, $a=0.14789903570478$, $\lambda=0.8732372308$, $K_m=3.819456626$;

 b) $m=7$, $a=0.077390319202550$, $\lambda=0.9365803433$, $K_m=8.987509817$.

\begin{figure}[h] 
  \centering
  \includegraphics[bb=0 0 693 308,width=5.02in,height=2.23in,keepaspectratio]{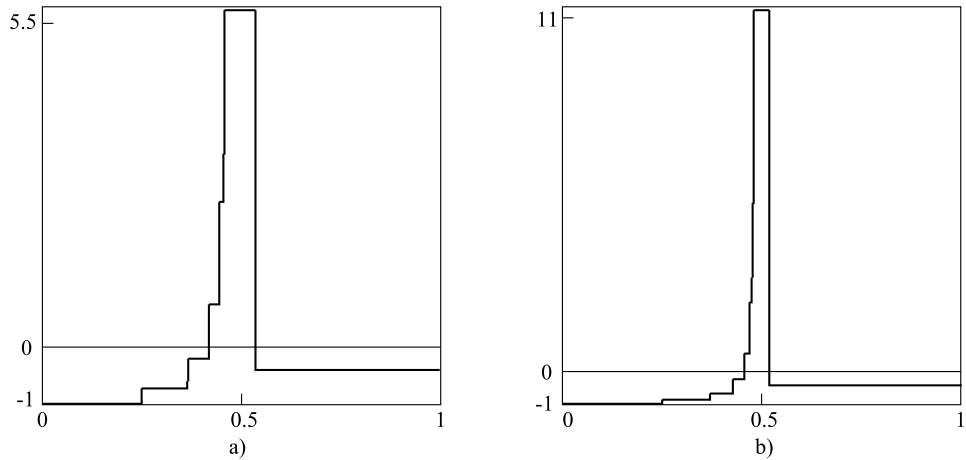}
  \caption{Normalized eigenfunctions $h_m$.}
  \label{fig:eigenfunctions}
\end{figure}

Note that the vertical scales of the pictures are very different.

\end{document}